\RequirePackage{fix-cm}

\documentclass{svjour3}                     

\smartqed  
\usepackage{graphicx}

\usepackage[english]{babel}
\usepackage{amsmath}
\usepackage{amsfonts}
\usepackage{amssymb}

\usepackage{changes}

\usepackage{pgf,tikz}
\usetikzlibrary{arrows}

\DeclareMathOperator{\arccot}{arccot}

\spnewtheorem*{pmp}{Pontryagin's Maximum Principle}{\bf}{\it}

\numberwithin{equation}{section}

\begin{document}

\title{Extreme properties of curves with bounded curvature on a sphere
}


\author{Alexander Borisenko\and Kostiantyn Drach}


\institute{A. Borisenko \at
	National Academy of Sciences of Ukraine,\\
	\email{aborisenk@gmail.com}           
           \and
           K. Drach \at
	Geometry Section, V.N. Karazin Kharkiv National University, Svobody Sq. 4, 61022, Kharkiv, Ukraine\\
	Tel.: +38-057-7075492\\
	\email{drach@karazin.ua}
}

\date{Received: date / Accepted: date}

\dedication{In memory of our friend and colleague V.~I.~Diskant}

\maketitle

\begin{abstract}
We give a sharp lower bound on the area of the domain enclosed by an embedded curve lying on a two-dimensional sphere, provided that geodesic curvature of this curve is bounded from below. Furthermore, we prove some dual inequalities for convex curves whose curvatures are bounded from above. 

\end{abstract}

\section*{Introduction}

The classical isoperimetric property of a circle in the two-dimensional space of constant curvature equal to $c$ claims that among all simple closed curves of a fixed length the maximal area is enclosed only by a circle. This property can be reformulated in an equivalent way in the form of an isoperimetric inequality. For an arbitrary simple closed curve of the length $L$ that encloses the domain of the area $A$ the following inequality holds (see, for example,~\cite{BZ})   
\begin{equation}
\label{isoineq}
L^2 - 4\pi A + c A^2 \geqslant 0,
\end{equation}
and the equality is attained only by circles.

Inequality~(\ref{isoineq}) gives a sharp upper bound on the area of the domain bounded by a curve provided that its length is fixed. At the same time, there exist simple closed curves that bound domains whose areas are arbitrary close to zero.    

Therefore, in order to obtain some meaningful estimates from below of the bounded area we need to further restrict the class of curves under consideration. One of the natural ways to do this is to consider curves of bounded curvature. Such class appeared in a number of extremal problems (see, for example,~\cite{AB,BorDr1,BorDr2,HTr,H,M}, and also~\cite{PZh}). In particular, in~\cite{HTr} the authors gave a bound on the area of domains enclosed by closed embedded plane curves of the fixed lengths whose curvatures $k$ satisfy the inequality $|k| \leqslant 1$ and the lengths satisfy some additional restrictions. Another type of bounds can be obtained if we consider the class of curves of curvatures bounded from below. In~\cite{BorDr1} for closed embedded plane curves whose curvatures $k$ in the generalized sense (see the definition below) satisfy the inequality $k \geqslant \lambda$ for some constant number $\lambda$ it was proved the following

\begin{theopargself}
\begin{theorem}[\cite{BorDr1}]
\label{ec2dth1} 
Let $\gamma$ be a closed curve embedded in the Euclidean plane  $\mathbb E^2$. If the curvature $k$ of $\gamma$ is bounded from below by some positive constant $\lambda$, then for the length $L$ of $\gamma$ and the area $A$ of the domain enclosed by $\gamma$ the following inequality holds: 
\begin{equation}
\label{ec2dth1eq1}
A \geqslant \frac{L}{2\lambda} - \frac{1}{\lambda^2}\sin\frac{\lambda L}{2}.
\end{equation}
Moreover, the equality case holds only for a \textmd{lune}, that is a boundary curve of the intersection of two domains enclosed by circles of curvature equal to $\lambda$.
\end{theorem}
\end{theopargself}

In the present paper we generalize Theorem~\ref{ec2dth1} for curves lying on a two-dimensional sphere.

\section{Preliminaries and the main result}

Let $\mathbb S^2 (k_1^2)$ be a standard two-dimensional sphere of Gaussian curvature equal to $k_1^2$ with $k_1 > 0$. In order to state the main result we need the following

\begin{definition}
A locally convex curve $\gamma \subset \mathbb S^2(k_1^2)$ is called \textit{$\lambda$-convex} with $\lambda \geqslant 0$ if for every point $P \in \gamma$ there exists a curve $\mu_P$ of constant geodesic curvature equal to $\lambda$ passing through $P$  in such a way, that in a neighborhood of $P$ the curve $\gamma$ lies from the convex side of $\mu_P$. 
\end{definition}

By the definition above, $0$-convex curves are just locally convex. If $\lambda>0$, then for every point on a curve there exists a locally supporting circle of curvature equal to $\lambda$.

We should also note that at  $C^k$-regular points of $\gamma$ with $k\geqslant 2$ the condition of being $\lambda$-convex is equivalent to the condition that at such points the geodesic curvature $\kappa_g$ of $\gamma$ satisfies the inequality $\kappa_g \geqslant \lambda$. Hence the class of $\lambda$-convex curves is a non-regular extension for the class containing smooth curves of geodesic curvature bounded from below by $\lambda$.

It is known that a convex curve is twice continuously differentiable almost everywhere and thus its geodesic curvature is almost everywhere well-defined. Therefore, a convex curve is $\lambda$-convex if and only if the inequality $\kappa_g \geqslant \lambda$ is satisfied at all points where geodesic curvature is defined. Even more, the set of a curve's \textit{vertexes}, that is points at which left and right semi-tangents do not coincide, is no-more-than countable.

We recall that $\gamma \subset \mathbb S^2(k_1^2)$ is a \textit{closed embedded curve} if it is homeomorphic to a circle and has no self-intersections. Embedded curves may not be smooth in general.

Note that a closed embedded $\lambda$-convex curve on a sphere is globally convex. Furthermore, if $\lambda > 0$, then at any point of this curve there exist a locally supporting circle, and thus the curve is globally strictly convex.

Observe also that a geodesic digon being a $0$-convex lune can enclose a domain with area arbitrary close to zero. Thus in order to get non-trivial estimates for the enclosed area, everywhere below we will assume $\lambda > 0$.

\begin{definition}
\textit{$\lambda$-convex polygon} is a closed embedded $\lambda$-convex curve composed of circular arcs of geodesic curvature equal to $\lambda$.
\end{definition}

\begin{definition}
A $\lambda$-convex polygon composed of two circular arcs we will call a \textit{$\lambda$-convex lune} or simply a \textit{lune}.
\end{definition}

We are now ready to formulate the main result of the paper that is summarized in the following theorem.

\begin{theorem}
\label{mainth}

Let $\gamma$ be a closed embedded $\lambda$-convex curve lying on a two-dimensional sphere $\mathbb S^2 (k_1^2)$ of Gaussian curvature equal to $k_1^2$. If $L(\gamma)$ is the length of $\gamma$ and $A(\gamma)$ is the area of the domain enclosed by $\gamma$, then
\begin{equation}
\label{maintheq1}
A(\gamma) \geqslant \frac{4}{k_1^2} \arctan\left(\frac{\lambda}{\sqrt{\lambda^2+k_1^2}}\tan\left(\frac{\sqrt{\lambda^2+k_1^2}}{4}L(\gamma)\right)\right)-\frac{\lambda}{k_1^2} L(\gamma).
\end{equation}
Moreover, the equality case holds only for $\lambda$-convex lunes. 
\end{theorem}

\begin{remark}
Observe that when $k_1 \to 0$ the right side of inequality~(\ref{maintheq1}) tends to the right side of inequality~(\ref{ec2dth1eq1}).
\end{remark}

It is important to note that the inequality in Theorem~\ref{mainth} express the \textit{isoperimetric property} of $\lambda$-convex lunes. To make the statement above precise, we need to reformulate Theorem~\ref{mainth} in the following equivalent way.

\begin{theorem}
\label{mainthalt}

Let $\gamma$ be a closed embedded $\lambda$-convex curve lying on a two-dimensional sphere $\mathbb S^2 (k_1^2)$. If $\gamma_0 \subset \mathbb S^2(k_1^2)$ is a $\lambda$-convex lune such that $$L(\gamma) = L(\gamma_0),$$
then $$A(\gamma) \geqslant A(\gamma_0)$$ and the equality case holds if and only if $\gamma$ is a $\lambda$-convex lune.
\end{theorem}

We will prove the main result in the form of Theorem~\ref{mainthalt} and after that show its equivalence to Theorem~\ref{mainth}.
 
\begin{remark}
The isoperimetric property of $\lambda$-convex curves, similar to the one proved for the Euclidean plane and the one that will be proved for a sphere, also holds for a two-dimensional Lobachevsky space $\mathbb H^2 (-k_1^2)$ for any $\lambda > 0$. 
\end{remark}

\section{Proof of the main result}
\label{mainres}

The principal tool for proving the main result will be Pontryagin's Maximum Principle. Let us recall all necessary definitions and a statement for a special case of this principle adapted for our needs (for a general case see~\cite[\S 1.4]{MDmO}, or equivalent but ideologically different approach in~\cite[\S 5]{ArMIT}). 

\begin{definition}
\label{contrpros}
A pair of functions $(x(t), u(t))$, defined as $(x(t), u(t)) \colon [t_0,t_1] \to \mathbb R^2 \times \mathbb R$ on some fixed segment $[t_0,t_1] \subset \mathbb R$ later denoted as $\Delta$, is a \textit{controlled process} if \textit{phase variable} $x(t)$, which is a pair of coordinate functions  $(x_1(t),x_2(t))$, is absolutely continuous function on $\Delta$, a \textit{control} $u(t)$ is bounded measurable function on $\Delta$, and the pair $(x(t), u(t))$ satisfies a \textit{controlled system}
\begin{equation}
\label{pmpeq1}
\begin{aligned}
\dot x (t) &= f (x(t),u(t))  \\
u(t) & \in \Delta_u 
\end{aligned}
\,\, \text{ a.e. on }\Delta,
\end{equation}   
where $\Delta_u \subset \mathbb R$ is a fixed segment, the function $f \colon \mathbb R^2 \times \mathbb R \to \mathbb R^2$ and its derivatives $f_x$ are continuous with respect to all variables; by the dot sign, as usual, we denote a derivative with respect to the variable $t$.
\end{definition}

For controlled processes we pose a minimization problem for the following functional 
\begin{equation}
\label{pmpeq2}
\mathcal J(x(\cdot),u(\cdot)) = \int\limits_{t_0}^{t_1} F_0 (x(t),u(t)) dt \to min,
\end{equation}
with an integral constraint
\begin{equation}
\label{pmpeq3}
\int \limits_{t_0}^{t_1} F_1 (x(t),u(t)) dt = A
\end{equation}
and periodic boundary conditions 
\begin{equation}
\label{pmpeq4}
x(t_0) = x(t_1) = a,
\end{equation}
where the function $F_0 \colon \mathbb R^2 \times \mathbb R \to \mathbb R$ is continuously differentiable on its domain, the function $F_1 \colon \mathbb R^2 \times \mathbb R \to \mathbb R$ is continuous together with its derivatives $(F_1)_x$, and $A \in \mathbb R$, $a \in \mathbb R^2$ are some fixed constants.

A \textit{trajectory} ${\{(x(t), u(t)) \colon t \in \Delta\}}$ that corresponds the controlled process~(\ref{pmpeq1}) is  \textit{admissible} for problem~(\ref{pmpeq2}) if for this trajectory constraints~(\ref{pmpeq3}),~(\ref{pmpeq4}) are satisfied. An admissible trajectory is \textit{optimal} if it minimizes the value of the functional $\mathcal J$ among all admissible trajectories.

Let us introduce a \textit{Pontryagin's function}
$$
\mathcal H(x,u,p,\lambda_0,\lambda_1) = p \cdot f(x,u) + \lambda_1 F_1 (x,u) - \lambda_0 F_0 (x,u).
$$
We will say that an admissible trajectory ${\{(x(t), u(t)) \colon t \in \Delta\}}$ for problem~(\ref{pmpeq2}) satisfies the~\textit{Pontryagin's maximum principle} if there are real numbers $\lambda_0, \lambda_1$ and absolutely continuous on $\Delta$ function $p(t) \colon \Delta \to \mathbb R^2$ such that for them the following conditions are met: 

\begin{enumerate}
\item[(i)] 
$\lambda_0 \geqslant 0$ (nonnegativity condition);

\item[(ii)] 
$\lambda_0 + |\lambda_1| + |p(t)| \neq 0$ for all $t \in \Delta$  (non-triviality condition);

\item[(iii)] 
$\dot p(t) = - \mathcal H'_x(x(t),u(t),p(t),\lambda_0,\lambda_1)$ a.e. on $\Delta$ (adjoint system);

\item[(iv)] 
$\max\limits_{u \in \Delta_u} \mathcal H(x(t),u,p(t),\lambda_0,\lambda_1) = \mathcal H(x(t),u(t),p(t),\lambda_0,\lambda_1)$ for almost all $t \in \Delta$ (maximality condition).
\end{enumerate}

Finally, we have the following theorem.

\begin{pmp}
\label{pmp}
Let ${\{(\hat x(t), \hat u(t)) \colon t \in \Delta\}}$ be an optimal trajectory for problem~(\ref{pmpeq2}). Then it satisfies the Pontryagin's maximum principle.
\end{pmp}

Now we have the necessary tool to proceed with proving the main result of the paper.

In order to use Pontryagin's Maximum Principle we need to construct a controlled system~(\ref{pmpeq1}). For this purpose let us introduce a so-called support function. Let $O \in \mathbb S^2(k_1^2)$ be a point inside a convex domain bounded by the curve $\gamma$. Let us consider on $\mathbb S^2(k_1^2)$ the polar coordinate system with the origin at the point $O$ and with the angular parameter $\theta$ with $\theta \in [0, 2\pi)$. For each geodesic ray $OL$, emanating from $O$ and forming an angle $\theta$ with some fixed direction, let us consider a geodesic perpendicular to $OL$ and that is supporting for the curve $\gamma$ at some point $P$. By strict convexity of $\gamma$ such geodesic always exists, and the point $P$ is unique. Denote $h(\theta)$ to be the distance from the point $O$ to the supporting geodesic above, measured along the ray $OL$. The function $h(\theta) \colon [0, 2\pi) \to [0, \pi/k_1)$ is a \textit{support function} for the curve $\gamma$.

We should note here that a convex curve is uniquely determined by its support function. 

The next proposition shows a relationship between a support function and the \textit{radius of curvature} of a curve $\gamma \subset \mathbb S^2(k_1^2)$ later denoted by $R$. We remind that, by definition, $R(\theta) = 1/k(\theta)$, where $k$ is the geodesic curvature of $\gamma$. For our curve we will call a \textit{contact radius of curvature} the following quantity 
\begin{equation*}
g(\theta)=\frac{1}{k_1}\tan\left(k_1 h(\theta)\right).
\end{equation*}

\begin{remark}
\label{remcoordsys}
Observe that for $\lambda$-convex curves we can always choose the origin of the polar coordinate system in such a way, that $$h(\theta) \leqslant \frac{1}{k_1}\arccot \frac{\lambda}{k_1} < \frac{\pi}{2 k_1}$$ or, equivalently, $g (\theta) \leqslant {1}/{\lambda}$ for all $\theta \in [0, 2\pi)$. We assume that everywhere below a coordinate system is chosen in the described way.
\end{remark}

Using the notion of a contact radius of curvature for $\lambda$-convex curves on a sphere we can prove the following proposition.

\begin{proposition}
\label{propsupfunc}

Let $\gamma \subset \mathbb S^2(k_1^2)$ be a closed embedded $\lambda$-convex curve on the sphere. If $R(\theta)$ and $g(\theta)$ are respectively the radius and the contact radius of curvature for $\gamma$, then
\begin{equation}
\label{propsupfunceq1}
R = \frac{g'' + g}{\left(1+\frac{ k_1^2 {g'}^2}{1 + k_1^2 g^2}\right)^\frac{3}{2}} \text{ for almost all $\theta \in [0,2\pi]$}
\end{equation}
(here by the prime sign we denote a derivative with respect the variable $\theta$).
\end{proposition} 

\begin{remark}
If in~(\ref{propsupfunceq1}) we take a limit with $k_1  \to 0$, we will get the classical formula $R = h'' + h$, which relates a support function and the radius of curvature of a plane curve.
\end{remark}

\begin{proof}

Without loss of generality we may assume $k_1 = 1$, and let us denote $\mathbb S^2(1)$ as $\mathbb S^2$.

We start with introducing the family $l_\gamma \subset \mathbb S^2$ of all supporting geodesic lines for $\gamma$. For each geodesic $l \in l_\gamma$ let $\xi(l)$ be an outward unit normal for $l$ with respect to the curve $\gamma$. Note that if at the point $l \cap \gamma$ the geodesic $l$ is tangent to $\gamma$, then $\xi(l)$ will also be an outward normal for $\gamma$.

Let us consider a map generated by unit normals $\xi(l)$ as follows:
\begin{equation}
\label{polarmap}
\begin{aligned}
\xi \colon &l_\gamma \to \mathbb S^{2}, \\
&l \mapsto \xi(l).
\end{aligned}
\end{equation}
This map is called the \textit{polar map} of the curve $\gamma$. The curve $\gamma^\ast \subset \mathbb S^{2}$ defined as $\gamma^\ast = \xi (l_\gamma)$ is the \textit{polar image} of $\gamma$ (or the \textit{dual curve}).   

Since $\gamma$ is convex, its polar map is a bijection between the set $l_\gamma$ and the set of points on $\gamma^\ast$. Moreover, $(\gamma^\ast)^\ast = \gamma$.

It appears that defining a curve using a support function is equivalent to defining its dual curve. To explain this statement, let $(t, \theta)$ be the coordinates in our polar coordinate system. From the definition of the polar map~(\ref{polarmap}) it easily follows that, assuming both $\gamma$ and $\gamma^\ast$ lies on the same sphere, the polar image $\gamma^\ast$ will be given by the equation
\begin{equation}
\label{tfordual}
t = \frac{\pi}{2} + h(\theta).
\end{equation}

From~(\ref{tfordual}) we obtain that the polar image of a circle with the radius $r$ will be a circle with the radius $\frac{\pi}{2} - r$. Note that the initial circle and its polar image have mutually reciprocal geodesic curvatures.

Using the observation above, we can deduce one important property of a polar map~(\ref{polarmap}). Namely, if $\gamma$ is a closed embedded $\lambda$-convex curve, then $\gamma^\ast$ will be a closed embedded curve through each point $P^\ast$ of which passes a circle of the curvature ${1}/{\lambda}$ such that in some neighborhood of $P^\ast$ this circle lies inside the closed domain bounded by $\gamma^\ast$. This property instantly implies that the curve $\gamma^\ast$ is in fact $C^{1,1}$-smooth, which in turn, by using~(\ref{tfordual}), implies $C^{1,1}$-smoothness of the support function $h(\theta)$.

Hence $\gamma^\ast$ is almost everywhere $C^2$-smooth, and thus the geodesic curvature $\kappa_g^\ast$ of $\gamma^\ast$ is almost everywhere well-defined. Let us show that for all values $\theta$ from $[0, 2\pi)$, where $\kappa^\ast_g$ is defined (we assume $\kappa_g^\ast \geqslant 0$), it is equal to the radius of curvature $R(\theta)$ of the curve $\gamma$, that is
\begin{equation}
\label{polarcurvprop}
\kappa_g^\ast(\theta) = R(\theta) \text{ a.e. on } [0, 2\pi).
\end{equation}

For regular curves this fact is well-known, but by the lack of a direct reference we will prove it here for completeness of exposition.  

We may assume that the curve $\gamma^\ast \subset \mathbb S^2$ with $\mathbb S^2 \subset \left(\mathbb E^3,\left<\cdot,\cdot\right> \right)$ is given in an arc-length parametrization $s^\ast$ by the unit radius-vector $r(s^\ast)$. Set $\xi^\ast$ to be an inward unit normal to $\gamma^\ast$.

Since all vectors under consideration are of the unit length, and since $\gamma^\ast$ lies on a sphere, we have
\begin{equation}
\label{vectid1}
\left<r',\xi^\ast\right> = \left<r,\xi^\ast\right> =  0,
\end{equation}
\begin{equation}
\label{vectid2}
\left<r', r\right> =  \left<{\xi^\ast}',\xi^\ast\right>= 0,
\end{equation}
where the prime sign denotes a derivative with respect to $s^\ast$.

From~(\ref{vectid1}) we obtain
\begin{equation}
\label{vectnorm}
\left<r,{\xi^\ast}'\right> = 0.
\end{equation}
Thus the vector $r$ is in fact a unit normal to the polar curve $\gamma$, treated as a curve in $\mathbb S^2$, and besides that $\xi^\ast$ is the radius-vector for $\gamma$.

Let us pick a point on $\gamma^\ast$ at which the geodesic curvature $\kappa_g^\ast$ is defined. For the rest of the proof of~(\ref{polarcurvprop}) we compute at this point. By definition
\begin{equation}
\label{geocurv}
\kappa_g^\ast = \left<r'',\xi^\ast \right>,
\end{equation}
Hence, using~(\ref{vectid1}), 
\begin{equation}
\label{firstterm1}
\left<r',{\xi^\ast}'\right> = -\kappa^\ast_g.
\end{equation}
From~(\ref{vectid2}) and~(\ref{vectnorm}) it follows that the vector ${\xi^\ast}'$ is collinear to the vector $r'$. Together with~(\ref{firstterm1}) this collinearity implies
\begin{equation}
\label{rodrig}
{\xi^\ast}' = -\kappa^\ast_g r'.
\end{equation}

If $s$ is the arc-length parameter on $\gamma$, then from~(\ref{rodrig}) we have
\begin{equation}
\label{secondterm}
\frac{ds}{ds^\ast} = \kappa_g^\ast, \,\,\, \frac{d\xi^\ast}{ds} = - r'.
\end{equation}
From the last relation it follows that the second derivative of $\xi^\ast$ is well-defined. Thus, using~(\ref{vectnorm}) and~(\ref{firstterm1}),
\begin{equation}
\label{firstterm2}
\left<{\xi^\ast}'',r\right>=\kappa^\ast_g.
\end{equation}

Finally, using~(\ref{firstterm1}) and~(\ref{secondterm}), with respect to the normal $r$ we compute
\begin{equation}
\label{geocurvpolar}
R = \left<\frac{d^2\xi^\ast}{{ds}^2}, r \right>^{-1} = \left<{\xi^\ast}'', r\right>^{-1}\left(\frac{ds}{ds^\ast}\right)^2 = \kappa_g^\ast,
\end{equation}
which proves~(\ref{polarcurvprop}).

Therefore, by property~(\ref{polarcurvprop}) in order to get formula~(\ref{propsupfunceq1}) it suffices to calculate the geodesic curvature of $\gamma^\ast$ given in a polar coordinate system with the metric $ds^2 = dt^2 + \sin^2 t \, d\theta^2$ by equation~(\ref{tfordual}). We should pay attention that, since the origin of the coordinate system and the polar curve lie in the distinct hemispheres, the geodesic curvature of $\gamma^\ast$ given by~(\ref{tfordual}) will be negative. Hence, assuming $\kappa_g^\ast \geqslant 0$, we should reverse the sign. After standard computations we will obtain
\begin{equation}
\label{propsupfunceq8}
R= \kappa^\ast_g = \frac{h'' \cos h + 2 h'^2 \sin h + \cos^2 (h) \sin h}{\left(h'^2 + \cos^2(h)\right)^{\frac{3}{2}}}.
\end{equation}

Substituting $h(\theta) = \arctan g(\theta)$ in~(\ref{propsupfunceq8}) we come to
$$
R = \frac{g'' + g}{\left(1 + \frac{g'^2}{1 + g^2}\right)^{\frac{3}{2}}},
$$
as required. Proposition~\ref{propsupfunc} is proved.
\qed
\end{proof}

In order to prove Theorem~\ref{mainthalt} we should fix the lengths of our curves and look for a minimum of the area of convex domains bounded by these curves. To formalize this problem, besides Proposition~\ref{propsupfunc}, we will need the expression for the length $L(\gamma)$ of a curve $\gamma$ and the expression for the area $A(\gamma)$ of the convex domain bounded by $\gamma$ in terms of its contact radius of curvature $g(\theta)$ and its radius of curvature $R(\theta)$. Without loss of generality we again may assume $k_1 = \lambda = 1$.

Using~(\ref{secondterm}),~(\ref{geocurvpolar}), and polar equation~(\ref{tfordual}) of the dual curve, we obtain
\begin{equation}
\label{mainproofeq0}
\frac{ds}{d\theta} = \frac{ds}{ds^\ast} \frac{ds^\ast}{d\theta} = R \sqrt{h'^2 + \cos^2(h)} =  R \frac{\sqrt{1 + g^2 + {g'}^2}}{1 + g^2}.
\end{equation}
Therefore, the length of $\gamma$ can be written in terms of $g$ and $R$ as follows:
\begin{equation}
\label{mainproofeq1}
L(\gamma) = \int \limits_0^{2\pi} ds(\theta) =  \int \limits_0^{2\pi} \frac{ds}{d\theta} d\theta =  \int \limits_0^{2\pi} R \frac{\sqrt{1 + g^2 + {g'}^2}}{1 + g^2} d\theta.
\end{equation}

Next, let us express the enclosed area $A(\gamma)$ in terms of $g(\theta)$. If $\{P_i \colon i \in I\} \subset \gamma$ is a no more then countable set of vertexes of the curve $\gamma$, then by the Gauss~-- Bonnet formula
\begin{equation}
\label{mainproofeq21}
A(\gamma) = 2\pi  - \int \limits_\gamma \frac{1}{R(s)} ds - \sum \limits_{i \in I} \varphi_i,
\end{equation}
where $\varphi_i$ are the jump angles of the tangent at the vertex $P_i$, $i\in I$. Using~(\ref{mainproofeq0}), we can rewrite~(\ref{mainproofeq21}) in the following way:
\begin{equation}
\label{mainproofeq22}
A(\gamma) = 2\pi - \int \limits_{[0,2\pi] \backslash \bigcup \limits_{i \in I} [\alpha_i, \beta_i]} \frac{\sqrt{1 + g^2 + {g'}^2}}{1 + g^2} d\theta - \sum \limits_{i \in I} \varphi_i,
\end{equation}
where the intervals $[\alpha_i, \beta_i]$ of values for $\theta$ correspond to the vertex $P_i$. Let us show that, in fact, for all $i \in I$
\begin{equation}
\label{mainproofeq23}
\int \limits_{\alpha_i}^{\beta_i} \frac{\sqrt{1 + g^2 + {g'}^2}}{1 + g^2} d\theta = \varphi_i.
\end{equation}
Indeed, from~(\ref{propsupfunceq1}) it follows that the contact radius of curvature of the vertex $P_i$ is equal to
\begin{equation}
\label{mainproofeq24}
g(\theta) = \tan(u_i) \cos(\theta - \theta_i),
\end{equation}
where $(u_i, \theta_i)$ are polar coordiantes of $P_i$, and $u_i$ is the distance from the origin $O$ to $P_i$. Substituting~(\ref{mainproofeq24}) into the left side of~(\ref{mainproofeq23}), we get
\begin{equation}
\label{mainproofeq25}
\begin{aligned}
\int \limits_{\alpha_i}^{\beta_i} \frac{\sqrt{1 + g^2 + {g'}^2}}{1 + g^2} d\theta &= \cos u_i \int  \limits_{\alpha_i}^{\beta_i} \frac{1}{\cos^2 u_i + \sin^2 u_i \cos^2(\theta - \theta_i)} d\theta\\
&=\left. \arctan \left(\cos u_i \cdot \tan(\theta - \theta_i)\right) \right|_{\alpha_i}^{\beta_i}.
\end{aligned}
\end{equation}  
Suppose $\varphi_{\alpha_i}$ and $\varphi_{\beta_i}$ are the angles between the corresponding to $\alpha_i$ and $\beta_i$ semi-tangents to $\gamma$ and the coordinate line $u = u_i$ at $P_i = (u_i,\theta_i)$; then
\begin{equation}
\label{mainproofeq26}
\varphi_i = \varphi_{\beta_i} - \varphi_{\alpha_i}.
\end{equation}
Moreover, straightforward computations show that $$\tan \varphi_{\alpha_i} = \cos u_i \cdot \tan (\alpha_i - \theta_i), \,\, \tan \varphi_{\beta_i} = \cos u_i \cdot \tan (\beta_i - \theta_i).$$ The last equalities together with~(\ref{mainproofeq25}) and~(\ref{mainproofeq26}) proves assertion~(\ref{mainproofeq23}).

Therefore, combining~(\ref{mainproofeq22}) and~(\ref{mainproofeq23}), we finally obtain
\begin{equation}
\label{mainproofeq27}
A(\gamma) = \int \limits_0^{2\pi} \left(1 - \frac{\sqrt{1 + g^2 + {g'}^2}}{1 + g^2} \right)d\theta.
\end{equation}

Thus we need to minimize $A(\gamma)$ taking into account~(\ref{propsupfunceq1}), and setting $L(\gamma) = L_0 = const$.
Let us interpret this problem as an optimal control problem with $t = \theta$, $u(t) = R(t)$, $x_1(t)=g(t)$, and $x_2(t) = \dot x_1(t) = g' (\theta)$. 

Since $\gamma$ is a $\lambda$-convex curve with $\lambda=1$, we have the restriction
\begin{equation}
\label{mainproofeq3}
0 \leqslant u(t) \leqslant 1 \text{ a.e. on }[0, 2 \pi].
\end{equation}

Taking into consideration~(\ref{mainproofeq3}), and rewriting~(\ref{propsupfunceq1}),~(\ref{mainproofeq1}), and~(\ref{mainproofeq27}) using the notations introduced above, we come to the following formal problem: 
\begin{equation}
\label{optcontrsys}
\begin{aligned}
&\int \limits_0^{2\pi} \left(1 - \frac{\sqrt{1 + x_1^2 + x_2^2}}{1 + x_1^2}\right) dt \rightarrow \min \\
&\int \limits_0^{2\pi} u \frac{\sqrt{1 + x_1^2 + x_2^2}}{1 + x_1^2} dt = L_0 \\
&\left\{
\begin{aligned}
&\dot x_1 = x_2 \\
&\dot x_2 = u \left(\frac{1 + x_1^2 + x_2^2}{1 + x_1^2}\right)^\frac{3}{2} - x_1  
\end{aligned} \right.\text{ a.e. on }[0, 2\pi]\\
&0 \leqslant u(t) \leqslant 1 \text{\, a.e. on }[0, 2\pi]\\
&x_1(0) = x_1 (2 \pi) \\ &x_2(0) = x_2 (2 \pi).
\end{aligned}
\end{equation}

Moreover, in problem~(\ref{optcontrsys}) the control $u(t)$ is bounded measurable function on $[0,2 \pi]$, and the phase variable $x(t)$ defined as $x(t) = (x_1(t), x_2(t))$ is absolutely continuous function on $[0,2\pi]$ since $h(\theta) \in C^{1,1}[0,2\pi]$ and $h(\theta) < \pi/2$ (see Remark~\ref{remcoordsys}). In addition, all the functions used in the functional, the integral constraint and the controlled system are continuous with respect to all variables. The same smoothness condition holds for derivatives with respect to $x$ of the mentioned functions. 

Therefore, the pair $(x,u)$ that satisfies the controlled system from~(\ref{optcontrsys}) is a controlled process in the sense of Definition~\ref{contrpros}, and if $(x,u)$ also satisfies  the integral constraint and the boundary conditions of problem~(\ref{optcontrsys}), then the corresponding trajectory $\{(x(t),u(t)) \colon t\in[0,2\pi]\}$ is an admissible trajectory.  

By Blaschke's selection theorem (see~\cite{Bla}) the posed problem of minimizing the area bounded by convex curves while keeping their lengths fixed has a solution. Hence the formalized version~(\ref{optcontrsys}) of the problem also has a solution. Thus in our case Pontryagin's Maximum Principle is a criterion for optimality of admissible trajectories.

If $\{\left(\hat x(t), \hat u(t)\right) \colon t \in [0, 2\pi]\}$ is an optimal trajectory for problem~(\ref{optcontrsys}), then by Pontryagin's Maximum Principle it must satisfy the conditions (i)~-- (iv). Let us rewrite them specifically for our problem. The maximum principle implies that there exist numerical multipliers $\lambda_0$, $\lambda_1 \in \mathbb R$ (with $\lambda_0 \geqslant 0$) and absolutely continuous on $[0,2\pi]$ functions $p_1(t)$ and $ p_2(t)$ such that they don't equal zero simultaneously, and the functions $p_i$ are the solutions of the adjacent system almost everywhere on $[0,2\pi]$, which in our case reads
\begin{equation}
\label{adjp1}
\dot p_1 = p_2 \left(1 + \frac{3 u x_1 x_2^2 \sqrt{1 + x_1^2 + x_2^2}}{\left(1 + x_1^2\right)^{\frac{5}{2}}}\right) + \frac{x_1 (\lambda_0 + \lambda_1 u) \left(1 + x_1^2 + 2 x_2^2\right)}{\left(1 + x_1^2\right)^2 \sqrt{1 + x_1^2 + x_2^2}},
\end{equation}
\begin{equation}
\label{adjp2}
\dot p_2 = -p_1 - p_2 \frac{3 u x_2 \sqrt{1 + x_1^2 + x_2^2}}{\left(1 + x_1^2\right)^{\frac{3}{2}}} - \frac{x_2 (\lambda_0 + \lambda_1 u)}{\left(1 + x_1^2\right) \sqrt{1 + x_1^2 + x_2^2}},
\end{equation}
 and Pontryagin's function, that for~(\ref{optcontrsys}) has the form
\begin{equation}
\label{ham}
\begin{aligned}
\mathcal H (x, u, p, \lambda_0, \lambda_1) &= p_1 x_2 + p_2 \left(  u \left(\frac{1 + x_1^2 + x_2^2}{1 + x_1^2}\right)^\frac{3}{2} - x_1\right) \\
&+ \lambda_1 \left(u \frac{\sqrt{1 + x_1^2 + x_2^2}}{1 + x_1^2}\right) - \lambda_0 \left(1 - \frac{\sqrt{1 + x_1^2 + x_2^2}}{1 + x_1^2}\right),
\end{aligned}
\end{equation}
satisfies the maximality condition   
\begin{equation}
\label{mincond}
\mathcal H(\hat x, \hat u, p,  \lambda_0, \lambda_1) = \max \limits_{u \in [0, 1]} \mathcal H(\hat x, u,  p, \lambda_0, \lambda_1)  \text{ a.e. on }[0, 2\pi].
\end{equation}

Let us analyze~(\ref{mincond}) in more details. In our case Pontryagin's function~(\ref{ham}) is linear with respect to $u$ and be written as $$\mathcal H = u \mathcal H_1 + \mathcal H_2,$$
where
\begin{equation}
\label{hamh1}
\mathcal H_1 = \lambda_1 \frac{\sqrt{1 + x_1^2 + x_2^2}}{1 + x_1^2} + p_2 \left( \frac{1 + x_1^2 + x_2^2}{1 + x_1^2}\right)^\frac{3}{2}.  
\end{equation}

From~(\ref{mincond}) it follows that the optimal control for problem~(\ref{optcontrsys}) must be of the form (here and below, for brevity, we will denote the the optimal solution $(\hat x, \hat u)$ as $(x, u)$)
\begin{equation}
\label{contr1}
u(t) = \left\{
\begin{aligned}
&1, \text{ for } \mathcal H_1 > 0 \\
&0, \text{ for } \mathcal H_1 < 0 \\
&\text{undefined}, \text{ for } \mathcal H_1 = 0
\end{aligned}
\right. \text{ a.e. on }[0, 2\pi].
\end{equation}

In order to determine the control completely, we need to consider a so-called \textit{singular trajectory}, that is an admissible for~(\ref{optcontrsys}) trajectory on which $\mathcal H_1$ is identically zero for some interval $(t_1,t_2) \subset [0,2\pi]$. For such trajectories the maximum principle doesn't give any information about control's behavior.

 From~(\ref{hamh1}) the condition $\mathcal H_1 = 0$ is equivalent to that on the singular trajectory
\begin{equation}
\label{singularextrp2}
p_2 = - \lambda_1 \frac{\sqrt{1 + x_1^2} }{1 + x_1^2 + x_2^2}.
\end{equation}

Substituting~(\ref{singularextrp2}) into differential equation~(\ref{adjp2}) and solving for $p_1(t)$ after all necessary cancellations we will get that on the singular trajectory
\begin{equation}
\label{singularextrp1}
p_1 =  x_2  \frac{\lambda_1 x_1 \sqrt{1 + x_1^2} - \lambda_0 \sqrt{1 + x_1^2 + x_2^2} }{(1 + x_1^2)(1 + x_1^2 + x_2^2)}.
\end{equation}

At the same time, the functions $p_1(t)$ and $p_2(t)$ have to satisfy remaining equation~(\ref{adjp1}) from the adjacent system. If we substitute $p_1$ and $p_2$ into~(\ref{adjp1}) and simplify it using the expressions for $\dot x_1$ and $\dot x_2$ from~(\ref{optcontrsys}), then we will come to the equality 
\begin{equation*}
\label{adjsextrcons}
\frac{\lambda_1 - \lambda_0 u}{\left(1 + x_1^2\right)^{\frac{3}{2}}}=0.
\end{equation*} 
Hence on the whole interval $(t_1,t_2)$ the equality 
\begin{equation}
\label{conscond}
\lambda_1 - \lambda_0 u(t) = 0
\end{equation}
must hold.

 If $\lambda_0 = 0$, then $\lambda_1=0$, and from~(\ref{singularextrp2}) and~(\ref{singularextrp1}) it follows that $p_1 \equiv p_2 \equiv 0$ on $(t_1,t_2)$. This contradicts the non-triviality condition (ii). Thus we may assume $\lambda_0 = 1$ and further consider only so-called \textit{normal trajectories} for which such a condition is met. With this in mind, equality~(\ref{conscond}) is possible only if $u(t) \equiv \lambda_1$ on the whole interval $(t_1,t_2)$. Since $\lambda_1$ is a constant real number and thus doesn't depend on the interval $(t_1,t_2)$, we conclude that the optimal control is equal $\lambda_1$ on any piece of the singular trajectory.

Therefore, (\ref{contr1}) can be rewritten as
\begin{equation}
\label{contr2}
u(t) = \left\{
\begin{aligned}
1,\,\, &\text{ for } \mathcal H_1 > 0 \\
0,\,\, &\text{ for } \mathcal H_1 < 0 \\
\lambda_1, &\text{ for } \mathcal H_1 = 0
\end{aligned}
\right. \text{ on }[0, 2\pi].
\end{equation}
Since the control can be redefined on a set of measure zero if necessary, we indeed may assume that~(\ref{contr2}) holds on the whole segment $[0,2\pi]$. Also, because $u(t) \in [0, 1]$, then $\lambda_1$ have to satisfy the inequality $1 \geqslant \lambda_1 \geqslant 0$. 

Note that from the geometric point of view the singular trajectory for~(\ref{optcontrsys}) corresponds to a circle of curvature equal to $1/\lambda_1$ if $\lambda_1 \neq 0$, and this circle is a set where the bang-bang control $u$ switches its value.

Let us show now that, in fact, in problem~(\ref{optcontrsys}) no arc of the singular extremal can be optimal. For this purpose let us consider the necessary condition for an arc of a singular trajectory to be a part of an optimal trajectory.

Recall that a natural number $q$ is an \textit{order} of the singular trajectory $(x^*,u^*)$ if it is the minimal number such that
$$\frac{\partial}{\partial u}\frac{d^{2q}}{dt^{2q}}\mathcal H_1 \left|_{(x^*,u^*)}\right. \neq 0, $$ 
where the time derivatives are taken with respect to the corresponding controlled system (see~\cite{KKM}).

An arc of a singular trajectory of order $q$ satisfies the so-called \textit{generalized Legendre~-- Clebsch condition} (see~\cite{KKM}) if along this arc
\begin{equation}
\label{LCC}
(-1)^q \frac{\partial}{\partial u} \left[\frac{d^{2q}}{dt^{2q}}\mathcal H_1\right] \leqslant 0.
\end{equation}

The necessary condition mentioned above states the following (see~\cite{KKM}, and also~\cite{B}): if an arc of a singular trajectory is a part of an optimal trajectory, then this arc must satisfy generalized Legendre~-- Clebsch condition. 

In our case we have the singular trajectory of order 1 ($q=1$). It can be easily verified by a straightforward computation. Using~(\ref{singularextrp2}),~(\ref{singularextrp1}), and the assumption $\lambda_0=1$, along the singular trajectory we have
$$
-\frac{\partial}{\partial u}\frac{d^{2}\mathcal H_1}{dt^{2}} = \frac{\left(1 + x_1^2 + x_2^2\right)^\frac{3}{2}}{\left(1 + x_1^2\right)^3} > 0. 
$$

The last inequality contradicts necessary condition~(\ref{LCC}). Hence inclusion of the singular trajectory in the solution of problem~(\ref{optcontrsys}) is not optimal. Therefore, on the segment $[0, 2\pi]$ we have a bang-bang control with only values 0 or 1.

From the geometric point of view we obtained that an optimal curve must consist of circular arcs of curvature equal to $1$. Thus the solution of~(\ref{optcontrsys}) belongs to the class of $1$-convex polygons with possibly infinite number of vertexes. It appears that for such a class we have the following geometric proposition, which is the discrete version of Theorem~\ref{mainthalt}.

\begin{proposition}
\label{mainthspecialcase}

Let $\gamma$ and $\tilde\gamma$ be respectively a $\lambda$-convex polygon and a $\lambda$-convex lune on the sphere $\mathbb S^2(k_1^2)$. If $$L(\gamma) = L(\tilde\gamma),$$
then $$A(\gamma) \geqslant A(\tilde\gamma).$$
Moreover, the equality holds if and only if $\gamma$ and $\tilde\gamma$ are congruent.

\end{proposition}

\begin{proof}

We will divide the proof into two steps: (I) when $\gamma$ is a centrally symmetric curve; (II) the general case. 

I. Suppose that $\gamma$ is a centrally symmetric $\lambda$-convex polygon. If at any point on $\gamma$ there exist a unique tangent geodesic, then $\gamma$ is a circle of constant curvature equal to $\lambda$. This is possible only when $L(\gamma) = {2\pi}/{\sqrt{k_1^2 + \lambda^2}}$. If $L(\gamma) < {2\pi}/{\sqrt{k_1^2 + \lambda^2}}$,  then the curve has at least one pair of centrally symmetric vertexes, that is points at which left and right semi-tangents do not coincide. If we have precisely one pair, then $\gamma$ is a lune.

Suppose that $\gamma$ has two distinct pairs of centrally symmetric vertexes $A$ and $\bar A$, $B$ and $\bar B$ (see Fig.~\ref{pic1} (a)).  

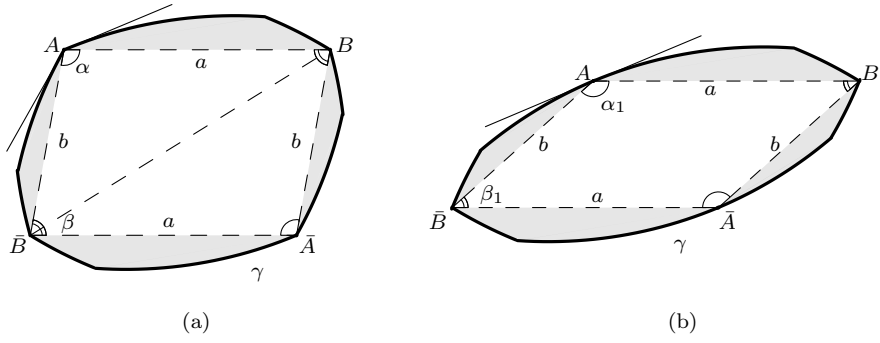
\begin{figure}[h]

\begin{center}

\begin{minipage}[H]{0.49\linewidth}
\begin{center}
\begin{tikzpicture}[line cap=round,line join=round,>=triangle 45,x=1.0cm,y=1.0cm,scale=0.25]
\clip(-5.77,-12.87) rectangle (13.24,2.78);
\fill[fill=black,fill opacity=0.1] (-2.38,0) -- (8.17,1.76) -- (11.62,0) -- cycle;
\fill[fill=black,fill opacity=0.1] (9.85,-9.84) -- (-0.71,-11.6) -- (-4.15,-9.84) -- cycle;
\fill[fill=black,fill opacity=0.1] (-4.82,-6.51) -- (-2.38,0) -- (-4.15,-9.84) -- cycle;
\fill[fill=black,fill opacity=0.1] (12.29,-3.33) -- (9.85,-9.84) -- (11.62,0) -- cycle;
\draw [shift={(-2.38,0)}] (0,0) -- (-100.22:0.83) arc (-100.22:0:0.83) -- cycle;
\draw [shift={(9.85,-9.84)}] (0,0) -- (79.78:0.83) arc (79.78:180:0.83) -- cycle;
\draw [shift={(11.62,0)}] (0,0) -- (180:0.83) arc (180:259.78:0.83) -- cycle;
\draw [shift={(-4.15,-9.84)}] (0,0) -- (0:0.83) arc (0:79.78:0.83) -- cycle;
\draw [dash pattern=on 6pt off 6pt] (-2.38,0)-- (11.62,0);
\draw [dash pattern=on 6pt off 6pt] (-2.38,0)-- (-4.15,-9.84);
\draw [dash pattern=on 6pt off 6pt] (-4.15,-9.84)-- (9.85,-9.84);
\draw [dash pattern=on 6pt off 6pt] (11.62,0)-- (9.85,-9.84);
\draw [shift={(6.57,-21.19)},line width=1.2pt,fill=black,fill opacity=0.1]  plot[domain=1.5:1.97,variable=\t]({1*23*cos(\t r)+0*23*sin(\t r)},{0*23*cos(\t r)+1*23*sin(\t r)});
\draw [shift={(-0.51,-19.54)},line width=1.2pt,fill=black,fill opacity=0.1]  plot[domain=1.02:1.18,variable=\t]({1*23*cos(\t r)+0*23*sin(\t r)},{0*23*cos(\t r)+1*23*sin(\t r)});
\draw [shift={(0.9,11.35)},line width=1.2pt,fill=black,fill opacity=0.1]  plot[domain=1.5:1.97,variable=\t]({-1*23*cos(\t r)+0*23*sin(\t r)},{0*23*cos(\t r)+-1*23*sin(\t r)});
\draw [shift={(7.97,9.7)},line width=1.2pt,fill=black,fill opacity=0.1]  plot[domain=1.02:1.18,variable=\t]({-1*23*cos(\t r)+0*23*sin(\t r)},{0*23*cos(\t r)+-1*23*sin(\t r)});
\draw(20.73,4.11) -- (24.89,4.11);
\draw(21.68,1.65) -- (25.84,1.65);
\draw [shift={(18,-3.68)},line width=1.2pt,fill=black,fill opacity=0.1]  plot[domain=3.26:3.41,variable=\t]({1*23*cos(\t r)+0*23*sin(\t r)},{0*23*cos(\t r)+1*23*sin(\t r)});
\draw [shift={(17.69,-11.24)},line width=1.2pt,fill=black,fill opacity=0.1]  plot[domain=2.63:2.93,variable=\t]({1*23*cos(\t r)+0*23*sin(\t r)},{0*23*cos(\t r)+1*23*sin(\t r)});
\draw [shift={(-10.22,1.4)},line width=1.2pt,fill=black,fill opacity=0.1]  plot[domain=2.63:2.93,variable=\t]({-1*23*cos(\t r)+0*23*sin(\t r)},{0*23*cos(\t r)+-1*23*sin(\t r)});
\draw [shift={(-10.54,-6.16)},line width=1.2pt,fill=black,fill opacity=0.1]  plot[domain=3.26:3.41,variable=\t]({-1*23*cos(\t r)+0*23*sin(\t r)},{0*23*cos(\t r)+-1*23*sin(\t r)});
\draw (-2.38,0)-- (3.3,2.4);
\draw (-2.38,0)-- (-5.39,-5.38);
\draw [shift={(11.62,0)}] (180:0.83) arc (180:259.78:0.83);
\draw [shift={(11.62,0)}] (180:0.62) arc (180:259.78:0.62);
\draw [shift={(-4.15,-9.84)}] (0:0.83) arc (0:79.78:0.83);
\draw [shift={(-4.15,-9.84)}] (0:0.62) arc (0:79.78:0.62);
\draw (-3.9,1.06) node[anchor=north west] {$A$};
\draw (11.5,1.06) node[anchor=north west] {$B$};
\draw (-5.7,-9.45) node[anchor=north west] {${\bar{B}}$};
\draw (9.55,-9.45) node[anchor=north west] {$\bar{A}$};
\draw (4.1,-0.05) node[anchor=north west] {${a}$};
\draw (2.4,-8.5) node[anchor=north west] {${a}$};
\draw (-3.1,-4) node[anchor=north west] {${b}$};
\draw (9.14,-4) node[anchor=north west] {${b}$};
\draw (-2.2,-0.3) node[anchor=north west] {$\alpha$};
\draw (-2.95,-8.32) node[anchor=north west] {$\beta$};
\draw (7.04,-11.12) node[anchor=north west] {$\gamma$};
\fill [color=black] (-2.38,0) circle (3.0pt);
\fill [color=black] (-4.15,-9.84) circle (3.0pt);
\fill [color=black] (9.85,-9.84) circle (3.0pt);
\fill [color=black] (11.62,0) circle (3.0pt);
\draw [dash pattern=on 6pt off 6pt] (-4.15,-9.84) -- (11.62,0);
\end{tikzpicture}
\end{center}
\end{minipage}
\hfill 
\begin{minipage}[h]{0.49\linewidth}
\begin{center}

\begin{tikzpicture}[line cap=round,line join=round,>=triangle 45,x=1.0cm,y=1.0cm,scale=0.25]
\clip(-11.17,-9.75) rectangle (12.87,2.86);
\fill[fill=black,fill opacity=0.1] (-2.38,0) -- (8.17,1.76) -- (11.62,0) -- cycle;
\fill[fill=black,fill opacity=0.1] (4.21,-6.71) -- (-6.35,-8.47) -- (-9.79,-6.71) -- cycle;
\fill[fill=black,fill opacity=0.1] (-8.29,-3.66) -- (-2.38,0) -- (-9.79,-6.71) -- cycle;
\fill[fill=black,fill opacity=0.1] (10.11,-3.05) -- (4.21,-6.71) -- (11.62,0) -- cycle;
\draw [shift={(-2.38,0)}] (0,0) -- (-137.85:0.83) arc (-137.85:0:0.83) -- cycle;
\draw [shift={(4.21,-6.71)}] (0,0) -- (42.15:0.83) arc (42.15:180:0.83) -- cycle;
\draw [shift={(11.62,0)}] (0,0) -- (180:0.83) arc (180:222.15:0.83) -- cycle;
\draw [shift={(-9.79,-6.71)}] (0,0) -- (0:0.83) arc (0:42.15:0.83) -- cycle;
\draw [dash pattern=on 5pt off 5pt] (-2.38,0)-- (11.62,0);
\draw [dash pattern=on 5pt off 5pt] (-2.38,0)-- (-9.79,-6.71);
\draw [dash pattern=on 5pt off 5pt] (-9.79,-6.71)-- (4.21,-6.71);
\draw [dash pattern=on 5pt off 5pt] (11.62,0)-- (4.21,-6.71);
\draw [shift={(6.57,-21.19)},line width=1.2pt,fill=black,fill opacity=0.1]  plot[domain=1.5:1.97,variable=\t]({1*23*cos(\t r)+0*23*sin(\t r)},{0*23*cos(\t r)+1*23*sin(\t r)});
\draw [shift={(-0.51,-19.54)},line width=1.2pt,fill=black,fill opacity=0.1]  plot[domain=1.02:1.18,variable=\t]({1*23*cos(\t r)+0*23*sin(\t r)},{0*23*cos(\t r)+1*23*sin(\t r)});
\draw [shift={(-4.74,14.48)},line width=1.2pt,fill=black,fill opacity=0.1]  plot[domain=1.5:1.97,variable=\t]({-1*23*cos(\t r)+0*23*sin(\t r)},{0*23*cos(\t r)+-1*23*sin(\t r)});
\draw [shift={(2.33,12.83)},line width=1.2pt,fill=black,fill opacity=0.1]  plot[domain=1.02:1.18,variable=\t]({-1*23*cos(\t r)+0*23*sin(\t r)},{0*23*cos(\t r)+-1*23*sin(\t r)});
\draw(20.73,4.11) -- (24.89,4.11);
\draw(21.68,1.65) -- (25.84,1.65);
\draw [shift={(11.52,-15.36)},line width=1.2pt,fill=black,fill opacity=0.1]  plot[domain=2.61:2.76,variable=\t]({1*23*cos(\t r)+0*23*sin(\t r)},{0*23*cos(\t r)+1*23*sin(\t r)});
\draw [shift={(6.65,-21.15)},line width=1.2pt,fill=black,fill opacity=0.1]  plot[domain=1.97:2.28,variable=\t]({1*23*cos(\t r)+0*23*sin(\t r)},{0*23*cos(\t r)+1*23*sin(\t r)});
\draw [shift={(-4.83,14.44)},line width=1.2pt,fill=black,fill opacity=0.1]  plot[domain=1.97:2.28,variable=\t]({-1*23*cos(\t r)+0*23*sin(\t r)},{0*23*cos(\t r)+-1*23*sin(\t r)});
\draw [shift={(-9.69,8.65)},line width=1.2pt,fill=black,fill opacity=0.1]  plot[domain=2.61:2.76,variable=\t]({-1*23*cos(\t r)+0*23*sin(\t r)},{0*23*cos(\t r)+-1*23*sin(\t r)});
\draw (-2.38,0)-- (3.3,2.4);
\draw (-2.38,0)-- (-8.05,-2.42);
\draw [shift={(11.62,0)}] (180:0.83) arc (180:222.15:0.83);
\draw [shift={(11.62,0)}] (180:0.62) arc (180:222.15:0.62);
\draw [shift={(-9.79,-6.71)}] (0:0.83) arc (0:42.15:0.83);
\draw [shift={(-9.79,-6.71)}] (0:0.62) arc (0:42.15:0.62);
\draw (-3.8,1.25) node[anchor=north west] {$A$};
\draw (11.3,1.25) node[anchor=north west] {${B}$};
\draw (-11.45,-6.5) node[anchor=north west] {${\bar{B}}$};
\draw (3.8,-6.5) node[anchor=north west] {$\bar{A}$};
\draw (3.05,0.15) node[anchor=north west] {${a}$};
\draw (-2.9,-5.5) node[anchor=north west] {${a}$};
\draw (-5.68,-2.5) node[anchor=north west] {${b}$};
\draw (6.5,-2.5) node[anchor=north west] {${b}$};
\draw (-2.4,-0.6) node[anchor=north west] {$\alpha_1$};
\draw (-8.75,-5.1) node[anchor=north west] {$\beta_1$};
\draw (1.43,-8) node[anchor=north west] {$\gamma$};
\fill [color=black] (-2.38,0) circle (3.0pt);
\fill [color=black] (-9.79,-6.71) circle (3.0pt);
\fill [color=black] (4.21,-6.71) circle (3.0pt);
\fill [color=black] (11.62,0) circle (3.0pt);
\end{tikzpicture}

\end{center}
\end{minipage}
\begin{minipage}[h]{1\linewidth}
\begin{tabular}{p{0.49\linewidth}p{0.49\linewidth}}
\vskip 0.05mm \centering (a) & \vskip 0.05mm \centering (b)
\end{tabular}
\end{minipage}

\caption{Centrally-symmetric $\lambda$-convex polygon $\gamma$ before (a) and after (b) the four-bar linkage isoperimetric deformation around the vertex $A$.}
\label{pic1}
\end{center}
\end{figure}

Let us consider the geodesic quadrilateral $AB\bar A\bar B$. Since $\gamma$ is centrally symmetric, the opposite sides of this quadrilateral are equal. Let us introduce the notations: $a = \left|AB\right| = \left|\bar A\bar B\right|$, $b = \left|A\bar B\right|=\left|\bar A B\right|$, $\alpha = \angle \bar B A B = \angle \bar B \bar A B$, $\beta = \angle A B \bar A = \angle A \bar B \bar A$. In the spherical geometry we have a fact, which is a complete analog of the elementary fact from the Euclidean geometry. 

 \begin{lemma}
\label{lemarea}

If $f(\alpha)$ is the area of a spherical triangle $BA\bar B$ on $\mathbb S^2(k_1^2)$ with $|BA| = a$, $|A\bar B| = b$ and with $\angle BA \bar B = \alpha$, then the function $f(\alpha)$ on the interval $[0,\pi]$ has a unique maximum point $\alpha_0$, and $f$ is strictly monotonous on $[0,\alpha_0)$ and $(\alpha_0, \pi]$ (see Fig.~\ref{pic1} (a)).
\end{lemma} 

\begin{proof}

Without loss of generality we may assume that $k_1=1$.

In the spherical case we can get the formula for the area of a triangle completely similar to the one in the Lobachevsky geometry (see~\cite[formula 2.9]{Wal}, and also~\cite{Bil}). In particular, if $g(\alpha) = \tan \frac{f(\alpha)}{2}$, then
\begin{equation*}
\label{lemareaeq1}
g(\alpha) = \frac{\sin a \sin b \sin \alpha}{(1 + \cos a)(1 + \cos b) + \sin a \sin b \cos \alpha}.
\end{equation*}

We compute:
$$
g'(\alpha) = \frac{\sin a \sin b \left( \cos \alpha (1 + \cos a)(1 + \cos b) + \sin a \sin b \right)}{\left((1 + \cos a)(1 + \cos b) + \sin a \sin c \cos \alpha \right)^2}.
$$

It is easy to see that on the interval $[0,\pi]$ the derivative $g'$ is equal zero if and only if $\cos(\alpha) = - \tan \frac{a}{2} \tan \frac{b}{2}$. Because the cosine is monotonous on $[0, \pi]$, the last equation has only one solution $\alpha_0$. Since $g(0)=g(\pi) = 0$, and $g \geqslant 0$, we obtain that $\alpha_0$ is a maximum point of the function $g(\alpha)$, and thus of the function $f(\alpha)$. The trivial observation that $f$ is  strictly monotonous on $[0,\alpha_0)$ and $(\alpha_0, \pi]$ finishes the proof of the lemma.
\qed
\end{proof}

From Lemma~\ref{lemarea} it follows that either $\alpha = \beta = \alpha_0$, or $\left(\alpha - \alpha_0\right)\left(\beta - \alpha_0\right) < 0$. In both cases there is an angle ($\alpha$ or $\beta$) whose increase leads to decrease of the area of the quadrilateral $AB\bar A \bar B$. Without loss of generality we can assume that this angle is $\alpha$. Note that increase of $\alpha$ implies decrease of $\beta$.

Now we introduce an isoperimetric deformation of the initial curve that preserves its $\lambda$-convexity, doesn't change the length, decreases the number of pairs of vertexes, and decreases the bounded area. For this we will use the idea of the four-bar linkage method due to Steiner (see~\cite{Bla}, and also~\cite{PR}).

Let us fix the arcs of the curve $\gamma$ between the points $A$, $B$, $\bar A$, and $\bar B$ and assume that at these points we have hinges. Since the four-bar linkage $AB\bar A \bar B$ is not rigid, the whole construction gains mechanical flexibility. Now, let us  increase the angle $ B A \bar B$ by rotating the links $AB$ and $A\bar B$ around the hinge $A$ until the angle between semi-tangents to $\gamma$ at $A$ will become equal to $\pi$ (see Fig.~\ref{pic1} (a, b)). Recalling that with such a move the area of the quadrilateral $AB\bar A \bar B$ decreases, and since we have fixed the arcs of $\gamma$, we obtain that the whole area bounded by $\gamma$ will decrease. At the same time the length of $\gamma$ will remain unchanged. Finally, since $\gamma$ is centrally symmetric, and, besides, the angle at the vertexes $B$ and $\bar B$ of $\gamma$ can only decrease, the deformed curve will be still $\lambda$-convex and centrally symmetric. 

Therefore, we have found an isoperimetric deformation of our $\lambda$-convex polygon into another $\lambda$-convex polygon such that this deformation decreases the bounded area and reduces the number of vertexes by $2$. Because by Blaschke's selection theorem there exists a centrally symmetric $\lambda$-convex polygon that bounds the least area among all such polygons with the same length, we get that it can be only a $\lambda$-convex lune. The case I of Proposition~\ref{mainthspecialcase} is proved.

II. We are ready to prove the general case. Let $PQ$ be a diameter of $\gamma$, that is a longest geodesic segment joining a pair of points on the curve. And let $\gamma_+$ and $\gamma_-$ be the arcs into which $P$ and $Q$ divide $\gamma$. We will denote $L_{+} = L(\gamma_{+})$ and $L_- = L(\gamma_{-})$. With such notations we have 
\begin{equation}
\label{lemspecialcaseeq2}
L(\gamma) = L_+ + L_-.
\end{equation}

Using $\gamma_+$ and $\gamma_-$ let us construct $\lambda$-convex polygons $\gamma_1$ and $\gamma_2$ by reflecting, respectively, $\gamma_+$ and $\gamma_-$ symmetrically with respect to the midpoint of the diameter $PQ$. From the first variation formula it follows that at $P$ and $Q$ there exist supporting for $\gamma$ geodesics such that the geodesic $PQ$ is orthogonal to them. This fact guaranties that the curves $\gamma_i$ are indeed $\lambda$-convex polygons. Observe that 
\begin{equation}
\label{lemspecialcaseeq3}
L(\gamma_1) = 2 L_+,\, L(\gamma_2) = 2 L_-.
\end{equation}

Let $\tilde \gamma_i$ be $\lambda$-convex lunes such that $L(\gamma_i) = L(\tilde \gamma_i)$ for $i \in \{1,2\}$. Hence for $\tilde \gamma_i$ and $\gamma_i$ we can apply the case I to obtain  
\begin{equation}
\label{lemspecialcaseeq1}
A(\gamma_i) \geqslant A(\tilde\gamma_i), \, i \in \{1,2\}.
\end{equation}

Let $\tilde A$ be an area bounded by a lune $\tilde \gamma$ considered as a function of its length $L(\tilde \gamma)$, that is $\tilde A \left(L (\tilde \gamma)\right) = A(\tilde \gamma)$. To proceed we have to know explicitly the function $\tilde A$ and its convexity properties. These facts are summarized in the following lemma, whose proof consists of a straightforward computation and thus omitted. 

\begin{lemma}
\label{lemluneest}
 If $\tilde \gamma \subset \mathbb S^2(k_1^2)$ is a $\lambda$-convex lune of length $L$, then
\begin{equation}
\label{lemluneesteq1}
\tilde A (L) = \frac{4}{k_1^2} \arctan\left(\frac{\lambda}{\sqrt{\lambda^2+k_1^2}}\tan\left(\frac{\sqrt{\lambda^2+k_1^2}}{4} L\right)\right)-\frac{\lambda}{k_1^2} L.
\end{equation}
Moreover, $\tilde A$ is a strictly convex function on the whole its domain of definition. 
\end{lemma}

Using strict convexity of $\tilde A$ from Lemma~\ref{lemluneest} and taking into consideration equalities~(\ref{lemspecialcaseeq2}),~(\ref{lemspecialcaseeq3}) and inequalities~(\ref{lemspecialcaseeq1}), we estimate:
\begin{equation*}
\begin{aligned}
A(\gamma) &= \frac{1}{2}\left(A(\gamma_1) + A(\gamma_2)\right) \geqslant \frac{1}{2}\left(A(\tilde \gamma_1) + A(\tilde \gamma_2)\right) \\ 
&= \frac{1}{2}\left(\tilde A(2 L_+) + \tilde A(2 L_-)\right) \geqslant \tilde A\left(\frac{2 L_+ +2 L_-}{2}\right) = \tilde A (L(\gamma)) = A(\tilde \gamma),
\end{aligned}
\end{equation*}
where $\tilde \gamma$ is the $\lambda$-convex lune of the same length as $\gamma$. 

To finish the proof of Proposition~\ref{mainthspecialcase} we need to consider the equality case. Note that in the chain of inequalities shown above the equality case, in a view of strict convexity of $\tilde A$ and since the case I, is possible only if $\gamma_1$ and $\gamma_2$ are congruent lunes. But this implies that the polygon $\gamma$ is a $\lambda$-convex lune too, which is equivalent to congruence of $\gamma$ and $\tilde \gamma$. Proposition~\ref{mainthspecialcase} is proved.  
\qed
\end{proof}

With the help of Proposition~\ref{mainthspecialcase} we are now ready to accomplish the proof of Theorem~\ref{mainthalt}. Indeed, by Proposition~\ref{mainthspecialcase} the only solution of our isoperimetric minimization problem for the bounded area in the class of $1$-convex polygons will be a $1$-convex lune. Hence, by Pontryagin's Maximum Principle it will also be the only solution of general problem~(\ref{optcontrsys}). Theorem~\ref{mainthalt} is proved. Theorem~\ref{mainth} directly follows from Theorem~\ref{mainthalt} and Lemma~\ref{lemluneest}.

\section{Dual results}

In this section with the help of polar maps we will obtain the dual results to the main results of our paper.

Let $\gamma \subset \mathbb S^2(k_1^2)$ be a $C^{1,1}$-smooth curve on a two-dimensional sphere. Consider some fixed vector field $X$ parallel along $\gamma$. Since a parallel field is a solution of a system of first order ODEs, smoothness of our curve will suffice to construct the field $X$. Let $\varphi(s)$ be the angle between $X$ and the tangent to $\gamma$ taken at the point corresponding to a arc-length parameter $s$ of $\gamma$. We will call \textit{upper geodesic curvature} $\overline{\kappa_g}$ and \textit{lower geodesic curvature} $\underline{\kappa_g}$ of $\gamma$ at a point corresponding to the arc length parameter $s_0$ the quantities
$$
\overline{\kappa_g} (s_0) = \limsup \limits_{s \rightarrow s_0} \frac{\varphi(s) - \varphi(s_0)}{s-s_0}, \,\,
\underline{\kappa_g} (s_0) = \liminf \limits_{s \rightarrow s_0} \frac{\varphi(s) - \varphi(s_0)}{s-s_0}.
$$  

Since $\gamma$ is a $C^{1,1}$-smooth curve, both its upper and lower geodesic curvatures are well defined, particularly, they are finite and don't depend on the choice of the vector field $X$ (see~\cite[p.~252]{doC}). If $\gamma$ is $C^m$-smooth with $m \geqslant 2$ at some point, then at this point $\overline{\kappa_g} = \underline{\kappa_g}=\kappa_g$.

Observe that if lower curvature of a curve satisfies the inequality $\underline{\kappa_g} \geqslant 0$, then this curve is locally convex. One more crucial observation we should make here is that if $\lambda \geqslant \overline{\kappa_g} \geqslant \underline{\kappa_g}\geqslant 0$ at each point on $\gamma$, then through any point $P$ on the curve passes a tangent to $\gamma$ circle of geodesic curvature equal to $\lambda$ such that in some neighborhood of $P$ this circle lies from the convex side of $\gamma$. The converse is also true, namely, an existence of a circle with curvature equal to $\lambda$ that touches a convex curve from its convex side (locally) at any point implies a boundedness of the curve's upper curvature. Furthermore, from the proof of Proposition~\ref{propsupfunc} it follows that convex curves with $\lambda \geqslant \overline{\kappa_g} \geqslant \underline{\kappa_g}\geqslant 0$ on a sphere $\mathbb S^2(k_1^2)$ are just polar images of $k$-convex curves with $k=k_1^2/\lambda$. Note that similarly to the above it is possible to define upper and lower curvatures for general convex curves. In such terminology $k$-convex curves on a sphere are precisely those with $\underline {\kappa_g} \geqslant k$.

In order to state the main result of this section, we need the following definition. A \textit{racetrack curve corresponding $\lambda$} with $\lambda > 0$ is a boundary of the convex hull for two equal circles of curvature equal to $\lambda$. Such curves appear as solutions of the maximization problem for the circumscribed radius of curve, provided that its length is fixed and geodesic curvature satisfies the inequality $1 \geqslant \underline{\kappa_g} \geqslant \overline{\kappa_g} \geqslant -1$ (see~\cite{H,AB} for the two-dimensional case, and~\cite{M} for curves in $\mathbb E^n$). 

Note that in the spherical case a racetrack curve corresponding $\lambda$ is just the polar image of a $1/\lambda$-convex lune.

Now we are ready to formulate the result dual to Theorem~\ref{mainth}.   

\begin{theorem}
\label{dualth}
Let $\gamma \subset \mathbb S^2(k_1^2)$ be a closed embedded curve whose upper and lower geodesic curvatures satisfy the inequality $0 \leqslant \underline{\kappa_g} \leqslant \overline{\kappa_g} \leqslant \lambda$ with some constant $\lambda \geqslant 0$. If $\tilde\gamma \subset \mathbb S^2(k_1^2)$ is a racetrack curve corresponding $\lambda$ such that
$$
A(\gamma) = A(\tilde\gamma),
$$
then
$$
L(\gamma) \leqslant L(\tilde\gamma).
$$
Moreover, equality holds if and only if $\gamma$ and $\tilde \gamma$ are congruent.
\end{theorem}

The stated theorem is equivalent to the following theorem.

\begin{theorem}
\label{dualthalt}

If $\gamma \subset \mathbb S^2 (k_1^2)$ is a closed embedded curve whose upper and lower geodesic curvatures satisfy the inequality $0 \leqslant \underline{\kappa_g} \leqslant \overline{\kappa_g} \leqslant \lambda$, then
\begin{equation}
\label{dualthsphcase}
\begin{aligned}
L(\gamma) &\leqslant \frac{2\pi}{k_1} + \frac{1}{\lambda} \left[2 \pi - k_1^2 A(\gamma)\right] - \\
&- \frac{4}{k_1} \arctan \left(\frac{k_1}{\sqrt{\lambda^2 + k_1^2}}\tan\left(\left[2\pi - k_1^2 A(\gamma)\right]\frac{\sqrt{\lambda^2 + k_1^2}}{4\lambda}\right)\right).
\end{aligned}
\end{equation}
Moreover, equality holds if and only if $\gamma$ is a racetrack curve corresponding $\lambda$.
\end{theorem}

\begin{proof}[Proof of Theorem~\ref{dualth}]
Theorem~\ref{dualth} follows from Theorem~\ref{mainth} and Theorem~\ref{mainthalt}. Indeed, if $\gamma$ and $\gamma^\ast$ are two polar curves on a sphere, then from the Gauss~-- Bonnet formula using the computations from the proof of Proposition~\ref{propsupfunc} it easily follows that 
\begin{equation}
\label{arealengthdual}
k_1 L(\gamma^\ast) = 2 \pi - k_1^2 A(\gamma).
\end{equation}
And since $\left(\gamma^{\ast}\right)^\ast = \gamma$, from~(\ref{arealengthdual}) we also have
\begin{equation}
\label{arealengthdual2}
k_1^2 A(\gamma^\ast) = 2 \pi - k_1 L(\gamma).
\end{equation}

Let $\gamma^\ast$ and $\tilde\gamma^\ast$ be the polar images of the curves $\gamma$ and $\tilde \gamma$.  As we showed above, the curve $\gamma^\ast$ will be a closed embedded ${k_1^2}/{\lambda}$-convex curve, and $\tilde\gamma^\ast$ will be a ${k_1^2}/{\lambda}$-convex lune. Using (\ref{arealengthdual}), and the assertion of the theorem, we obtain  
$$
k_1 L(\gamma^\ast) = 2 \pi - k_1^2 A(\gamma) = 2 \pi - k_1^2 A(\tilde \gamma) = k_1 ^2 L(\tilde\gamma^\ast).
$$
Thus by Theorem~\ref{mainthalt}, 
$$
A(\gamma^\ast) \geqslant A(\tilde\gamma^\ast),
$$
from which using~(\ref{arealengthdual2}), the inequality of Theorem~\ref{dualth} follows directly. The equality case is proved automatically. Theorem~\ref{dualth} is proved.
\qed
\end{proof}

Equivalence of Theorems~\ref{dualth} and~\ref{dualthalt} is obtained by a straightforward computation of the length of a racetrack curve corresponding $\lambda$ assuming it bounds a fixed area. Even more, inequality~(\ref{dualthsphcase}) easily follows from inequality~(\ref{maintheq1}) with the help of formula~(\ref{arealengthdual}), and the discussed above polar correspondence between $C^{1,1}$-smooth curves with $0 \leqslant \underline{\kappa_g} \leqslant \overline{\kappa_g} \leqslant \lambda$ on $\mathbb S^2(k_1^2)$ and $k_1^2/\lambda$-convex curves on $\mathbb S^2(k_1^2)$.

\begin{remark}
We can also get the dual version of Theorem~\ref{ec2dth1}. Indeed, on the Euclidean plane we can define convex curves of bounded upper ($\overline{\kappa}$) and lower ($\underline{\kappa}$) curvature, and race-track curves in the same way as it was done in the spherical case. But since on the Euclidean plane a polar map depends on the choice of a polar center, and hence there is no absolute polarity as we have on a sphere, the dual result on the plane is not an immediate corollary of Theorem~\ref{ec2dth1}. However, by applying Pontryagin's Maximum Principle once more similarly to~\cite{BorDr1} with curvature of a curve as a control parameter, it is easy to get the desired result. More precisely, if $\gamma \subset \mathbb E^2$ is a closed embedded curve with $0 \leqslant \underline{\kappa} \leqslant \overline{\kappa} \leqslant \lambda$ ($\lambda > 0$), then
$$
L(\gamma) \leqslant \lambda A(\gamma) + \frac{\pi}{\lambda},
$$
and the equality holds only for a race-track curve corresponding $\lambda$.
\end{remark}


\begin{acknowledgements}
The authors would like to thank an anonymous referee for useful comments and suggestions that helped to improve the exposition.
\end{acknowledgements}




\end{document}